\documentclass[12pt,a4paper]{article}
\usepackage{amsmath,amssymb,amsthm,amsfonts,amscd,euscript,verbatim, t1enc, newlfont}
\usepackage{hyperref}
\usepackage{yfonts}
\usepackage{graphicx}
\usepackage[all]{xy}
 \theoremstyle{plain}
\newtheorem{theo}{Theorem}[section]

\newtheorem{pr}[theo]{Proposition}

 \newtheorem{lem}[theo]{Lemma}
 \newtheorem{coro}[theo]{Corollary}

\theoremstyle{remark}
\newtheorem{rema}[theo]{Remark}

\theoremstyle{definition}
\newtheorem{defi}[theo]{Definition}

\newcommand \cu{\underline{C}}

\newcommand\ilim\varinjlim

\newcommand\obj{\operatorname{Obj}}

\newcommand\ab{{Ab}}

\newcommand\n{\mathbb{N}}
\newcommand\z{{\mathbb{Z}}}

\DeclareMathOperator\co{\operatorname{Cone}}
\DeclareMathOperator\mo{\operatorname{Mor}}

\begin{document}

\title
 {A Nullstellensatz for triangulated categories}
 \author{M.V. Bondarko\thanks{ 
 The first author was supported by RFBR 
grant no.   14-01-00393-a,  by  Dmitry Zimin's Foundation "Dynasty", and by the Scientific schools grant no. 3856.2014.1.
}, V.A. Sosnilo\thanks{
 The second author was  supported by the Chebyshev Laboratory
(Department of Mathematics and Mechanics, St. Petersburg State University)  under the RF Government grant 11.G34.31.0026, and also by the JSC "Gazprom Neft".
Both authors were supported by RFBR 
grant no. 15-01-03034-a. 
}}
\maketitle

\begin{abstract}
The main goal of this paper is to prove the following: for  a  triangulated
category  $ \underline{C}$ and $E\subset \operatorname{Obj} \underline{C}$ there exists a cohomological functor
$F$ (with values in some abelian category) 
such that
$E$ is 
its set of zeros if (and only if) $E$ is closed
with respect to retracts and extensions (so, we obtain a certain Nullstellensatz for functors of this type). Moreover, for
$ \underline{C}$ being an $R$-linear category (where $R$ is a commutative ring) this is also equivalent to the existence of 
an $R$-linear functor 
 $F: \underline{C}^{op}\to  R-\operatorname{mod}$ satisfying this property. 
As a corollary, we prove that an object $Y$ belongs to the corresponding
"envelope" of some $D\subset \operatorname{Obj}  \underline{C}$
whenever
the same is true for the images of $Y$ and $D$ in all the categories $ \underline{C}_p$ obtained from $ \underline{C}$ by means of
"localizing the coefficients" at  
maximal ideals $p \triangleleft R$. Moreover, to prove our theorem we develop certain new methods for relating triangulated categories to their (non-full) countable triangulated subcategories. 

The results of this paper can be
applied to 
 weight structures and
triangulated categories of motives.
\end{abstract}

\tableofcontents

\section*{Introduction}

Certainly, for any class of cohomological functors $\{F_i\}$ from a triangulated category $ \underline{C}$ (with values in some abelian categories $\underline{A}_i$) the  class $E=\{c\in \operatorname{Obj} \underline{C},\ F_i(c) = 0\ \forall i\}$  is extension-closed 
and Karoubi-closed in $ \underline{C}$ (i.e., it is  closed
with respect to retracts and "extensions").
In this paper we prove that (somewhat surprisingly) the converse statement is also true (assuming   that $\operatorname{Obj}  \underline{C}$ is a set to avoid set-theoretical difficulties).

Being more precise, we prove the following statement.

\begin{theo}\label{extsunc}
Let $ \underline{C}$ be a small $R$-linear triangulated category (where $R$ is a commutative unital ring); let $E$ be an extension-closed Karoubi-closed subset of $\operatorname{Obj}  \underline{C}$. 
Then for any $Y \in \operatorname{Obj} \underline{C}\setminus E$ 
there exists an $R$-linear ("separating") cohomological functor $F^Y: \underline{C}^{op}\to  R-\operatorname{mod}$ 
 such that $F^Y|_E = 0$ and $F^Y(Y)\neq \{0\}$. 
\end{theo}
Certainly, this result can be applied in the (most general) case $R=\z$. Thus (cf.  Corollary \ref{mcoro} below) $E\subset \obj \underline{C}$ is the set of zeroes of some collection of cohomological functors
if and only if it is extension-closed and Karoubi-closed; this a certain Nullstellensatz for functors of this type (whence the name of the paper).

Now we explain our motivation for studying this question.
For various triangulated categories the existence of certain distinguished classes of objects is very important; in particular, one is often interested 
in sets of zeros for certain collections of (co)representable functors (since such classes can yield $t$-structures, weight structures, and more generally Hom-orthogonal pairs; see Definition 3.1 in \cite{postov}, as well as \cite{talosa}, \cite{paukcomp}, and \cite{bgern}). Yet it 
is scarcely possible to describe in general all 
classes of objects that can be characterized this way  (for 
arbitrary families of representable functors; note still that quite important results related to this question were obtained in the aforementioned papers). This led the authors to modifying this description problem (and obtaining Theorem \ref{extsunc} as a complete answer to the corresponding question). 

This theorem has a corollary that seems to be interesting for itself. 
For $D\subset \operatorname{Obj} \underline{C}$ let us call a class $E\subset \operatorname{Obj}  \underline{C}$ 
the {\it envelope} of $D$ if it is the smallest subclass of $\operatorname{Obj}  \underline{C}$ that contains $D\cup  \{0\}$ and is closed with respect to retracts and extensions. 

We note the following: if $F: \underline{C}^{op} \to  R-\operatorname{mod}$ is a cohomological functor (where $R$ is a commutative unital ring) and $M$ is a flat $R$-module then
the ("naive tensor product") functor $F(-)\otimes M: \underline{C}^{op}\to  R-\operatorname{mod}$ is cohomological also. 
Considering these functors for $M=R_{p}$, where $p$ runs through all maximal ideas of $R$, we easily deduce the following statement. 

 \begin{coro}\label{clocoeff} 
 Let $ \underline{C}$ be an $R$-linear triangulated category. For any maximal ideal $p \triangleleft R$ 
denote by $R_p$ the localization of $R$ at $p$, and 
denote by $ \underline{C}_p$ the triangulated category whose objects are those of  $ \underline{C}$, and whose morphisms are obtained by tensoring the $ \underline{C}$-ones by $R_p$ (over $R$; see 
 Proposition \ref{trtens} below). 
Denote by $L_p$ the obvious functor $ \underline{C}\to  \underline{C}_p$.

Let $D\subset \operatorname{Obj}  \underline{C}$,  $Y\in \operatorname{Obj} \underline{C}$.
Then $Y$ belongs to the envelope $E$ of $D$ in $ \underline{C}$  if and only if 
$L_p(Y) $ belongs to  the envelope of $L_p(D)$ (in $ \underline{C}_p$) for every maximal ideal $p$ of $R$.
\end{coro}

Note that 
$L_p$ (for various $ \underline{C}$ any any prime ideal $p$ of $R$) is exactly "the natural localization of coefficients functor" (corresponding to passing from $R$-linear categories to $R_{p}$-linear ones); see Appendix A.2 of \cite{kellyth}  for a thorough study    of this construction in the case $R=\z$ and Proposition B.1.5 of \cite{cdet} for the general case. 
	
	We would also like to note that our motivation for considering the functors of the type $L_p$ along with the corollary stated was the following "motivic" one:  certain resolution of singularities statements needed for the study of various motivic categories (that were proved by Gabber) can only be applied "directly" to various $R_{p}$-linear categories of 
	motivic origin (
	whereas $R$ is usually a localization of $\z$). 
	Whereas the ("triangulated") properties of $L_p$ were sufficient for the "globalization of coefficients" arguments in ibid. and in \cite{kellyth} (since the 
	 study of the "triangulated" envelope of $ \cup_{i\in \z}D[i]$ for certain $D$ was sufficient for the aforementioned papers, and in this setting it suffices to apply the theory of Verdier quotients), for some other problems (including the one studied in \cite{bzp})  the consideration of  "general" envelopes is quite relevant.  

Now we give some more detail on the contents of the paper.


In \S\ref{snotata} we introduce some basic (and more or less common) categorical definitions.

In 
\S\ref{prcount} we 
prove Theorem \ref{extsunc} in the case where both $E$  
 and all the hom-sets in $ \underline{C}$ are countable. Our idea is to construct $F^Y$ as a (direct) limit of representable functors, whereas the construction of the corresponding objects is closely related to the "approximation" methods introduced in \S III.2 of \cite{bre} (that were also applied in the proof of the main statement of \cite{paukcomp}; the obvious duals of these arguments were used in the proof of Theorem 2.2.6 of \cite{bgern}).

In \S\ref{prgen} we introduce a certain (new) technique 
that relates general triangulated categories to 
 their subcategories with  countable sets of objects and morphisms. It easily yields the proof of the theorem in the case  $R = \z$; this proof also works for the $R$-linear case if $R$ is at most countable. Finally, in 
\S\ref{rcas} we prove the theorem in general (via approximating the ring $R$ by its countable subrings). Actually, the corresponding argument (that may possibly be interesting for itself, especially in conjunction with other methods of the section) was our main reason for formulating the main results for $R$-linear categories; yet the reader may certainly restrict himself to the case $R=\z$ (that seems to be sufficient for the motivic applications the authors have in mind) throughout the paper.


In 
\S\ref{applc} we prove Corollary \ref{clocoeff}.

 The authors are deeply grateful to prof. A.I. Generalov for his very useful comments.

\section{Notation and conventions}\label{snotata}

\begin{itemize}

\item Given a category $C$ and  $X,Y\in\operatorname{Obj} C$, we denote by
$C(X,Y)$ the set of morphisms from $X$ to $Y$ in $C$.


\item  $R$ is a commutative unital ring.

\item $ \underline{C}$ below will always denote an $R$-linear triangulated category; we will assume it to be small till \S\ref{applc}. 

\item $\underline{A}$ will be an abelian category.

\item
For categories $C$ and $C'$ we will write $H:C \to C'$ only if $H$ is a covariant functor.
So, 
we will say that $F$ is a cohomological functor from $ \underline{C}$ to $\underline{A}$ and write $F: \underline{C}^{op}\to \underline{A}$ if 
$F$ is contravariant on $ \underline{C}$ and sends $ \underline{C}$-distinguished triangles into long exact sequences.
Note yet that 
an alternative convention is often used for this matter; then such an  $F$ is called a cohomological functor from $ \underline{C}^{op}$ into $\underline{A}$
(cf. \cite{krause}).



\item For any  $A,B,C \in \operatorname{Obj} \underline{C}$ we will say that $C$ is an {\it extension} of $B$ by $A$ if there exists a distinguished triangle $A \to C \to B \to A[1]$ (in $ \underline{C}$).
 A class $E\subset \operatorname{Obj}  \underline{C}$ is said to be  {\it extension-closed}
    if it 
		is closed with respect to extensions and contains $0$. 
		
		\item 
		For $X,Y\in\operatorname{Obj}  \underline{C}$, we say that $X$ is a {\it
retract} of $Y$ 
 if $\operatorname{id}_X$ can be 
 factored through $Y$ (
since $ \underline{C}$ is triangulated, $X$ is a retract of $Y$ if and only if $X$ is its direct summand).
A 
class $E\subset \operatorname{Obj}  \underline{C}$
is said to be {\it Karoubi-closed}
  in $ \underline{C}$ if $E$
contains all $ \underline{C}$-retracts of its elements. 

		\item  We will call the smallest extension-closed Karoubi-closed subclass 
of $\operatorname{Obj} \underline{C}$ that  contains a given class $D\subset \operatorname{Obj} \underline{C}$ 
  the 
 {\it envelope} of $D$ (in $ \underline{C}$).

\item Given $f\in \underline{C} (X,Y)$ (where $X,Y\in\operatorname{Obj} \underline{C}$) we will call the third vertex
of (any) distinguished triangle $X\stackrel{f}{\to}Y\to Z$ a {\it cone} of
$f$ (recall that various choices of cones are isomorphic, but the choice of these isomorphisms is non-canonical).

\item A directed set is a set with a partial order such that every pair of elements has an upper bound. 

\item Let $(I,\succeq  )$ be a directed set and let $x\in I$. By $I_x$ we denote the set of 
$i\in I$ such that $i\succeq  x$. We will say that a certain property of $i\in I$ is fulfilled for sufficiently large $i$ if it is fulfilled for all $i\in I_x$ (for some $x\in I$).

\item We will assume that the set $\mathbb{N}$ of natural numbers starts from $1$. 

\end{itemize}

\section{Proof of the theorem in the countable case}\label{prcount}

In this section we prove the statement of 
Theorem \ref{extsunc} under certain (rather strong) countability restrictions. 

\begin{pr}\label{exts}
Suppose all hom-sets in $ \underline{C}$ are at most countable; let $E = \{e_i\}_{i\in \mathbb{N}}$ be 
a countable extension-closed and Karoubi-closed subset of $ \underline{C}$.
Then for any $Y \in \operatorname{Obj} \underline{C}\setminus E$ there exists an $R$-linear cohomological functor $F^Y: \underline{C}^{op}\to R-\operatorname{mod}$ 
  such that $F^Y|_E = 0$ and $F^Y(Y)\neq \{0\}$. 
\end{pr}

\begin{proof}
For every $Z\in \operatorname{Obj}  \underline{C},\ i\in \mathbb{N}$, we fix some surjection $\mathbb{N} \to  \underline{C}(e_i,Z)$.
We will denote by $f_{Z,i,j}$ the image of $j$ under this map.

Now we construct a certain inductive system $\{Y_i\in \operatorname{Obj}\underline{C}\}_{i\ge 0}$ along with the corresponding connecting morphisms $\alpha_{i}:Y_i\to Y_{i+1}$ (for all $i\ge 0$).
We start from $Y_0 = Y$.
Then we use induction to choose $Y_i$ and the corresponding morphisms.  
For $n>0$ assume we have chosen $Y_k\in \operatorname{Obj}\underline{C}$ for all $0\le k \le n-1$ along with 
$\alpha_k\in \underline{C}(Y_k,Y_{k+1})$ for  $0\le k \le n-2$. 
For $0\le k \le n-1$ denote by $\alpha_{k,n-1}$  the composition $\alpha_{n-2}\circ \cdots \circ \alpha_{k}$ (so, $\alpha_{n-1,n-1}= \operatorname{id}_{Y_{n-1}}$). 
We consider 
$$Y_n = \operatorname{Cone}(\bigoplus\limits_{0\le k\le n-1}\bigoplus\limits_{1\le j\le n}\bigoplus\limits_{1\le i\le n} e_i \xrightarrow{ p_{n-1} \circ (
\bigoplus( \alpha_{k,n-1}\circ f_{Y_{k},i,j})%
} Y_{n-1})$$
where $p_{n-1}$ is the projection morphism $\bigoplus\limits_{0\le k\le n-1}\bigoplus\limits_{1\le j\le n}\bigoplus\limits_{1\le i\le n}Y_{n-1} \to Y_{n-1}$.
To finish the inductive step we take for  $\alpha_{n-1}$ the morphism $Y_{n-1} \to Y_n$ coming from the definition of a cone.

We take $F^Y(-) = \ilim\limits_{n\in \mathbb{N}}  \underline{C}(-, Y_n)$ (with the connecting morphisms induced by the corresponding $\alpha_{k,n}$).
Certainly, $F^Y$ is an $R$-linear cohomological functor.

Now we verify that $F^Y$ 
fulfills the (remaining) conditions desired. Firstly we check that any element $x$ of $ \underline{C}(e_i,Y_n)$ (for some $i,n\in \n$) dies in $ \underline{C}(e_i,Y_N)$ for some $N > n$.
Indeed, there is a number $k$ such that $x = f_{Y_n, i,k}$. Consider $N = \max\{k,n+1\}$. By the definition of $Y_N$, there exists 
a distinguished triangle
$$
\begin{CD} 
e_i \oplus X@>{\begin{pmatrix}\alpha_{n,N-1}\circ x & x'\end{pmatrix}}>>      Y_{N-1}   @>{\alpha_{N-1}}>>     Y_N     @>{}>> (e_i \oplus  X)[1],
\end{CD}
$$
where
$X$ belongs to $E$ and $x'$ is a certain element of $ \underline{C}(X,Y_{N-1})$.
Hence the sequence 
$$\cdots \to  \underline{C}(e_i,e_i \oplus X) \to  \underline{C}(e_i,Y_{N-1}) \stackrel{(\alpha_{N-1})_*}\to  \underline{C}(e_i,Y_N) \to \cdots$$
is exact. 
The morphism $e_i \stackrel{\alpha_{n,N-1}\circ x}\to Y_{N-1}$ factors through $e_i \oplus X$, thus $\alpha_{n,N}\circ x = (\alpha_{N-1})_*(\alpha_{n,N-1}\circ x) = \alpha_{N-1} \circ\alpha_{n,N-1}\circ x = 0$. 

By 
the definition of 
the direct limit of functors, 
we obtain $F^Y(e_i)=\{0\}$ (for all $i\in \n$), i.e., $F^Y|_E = 0$. 

It remains to check whether $F^Y(Y) = \{0\}$. 
If this is the case, then $\operatorname{id}_Y\in  \underline{C}(Y,Y) =  \underline{C}(Y,Y_0)$ goes to zero in $ \underline{C}(Y,Y_N)$ for some  $N\in \mathbb{N}$.
Now, $\operatorname{Cone}(\alpha_{k})[-1]$ belongs to $E$ for any $k \ge 0$ by  construction. Hence  $\co(\alpha_{k,n})[-1]\in E$ for any $0\le k\le n$ (and $\alpha_{k,n}$ defined as above; here we apply the octahedral axiom).
We denote  $\co(\alpha_{0,N})[-1]$ by $X'$; applying the functor $ \underline{C}(Y,-)$ to the distinguished triangle  $X' \to Y \to Y_N \to X'[1]$ we  obtain a long exact sequence 
$$\cdots \to  \underline{C}(Y,X') \to  \underline{C}(Y,Y) \stackrel{(\alpha_{0,N})_*}\to  \underline{C}(Y,Y_N) \to \cdots$$
Since $(\alpha_{0,N})_*(\operatorname{id}_Y)$ is zero, $\operatorname{id}_Y$ factors through $X'$. 
So, $Y$ is a retract of an element of $E$; hence it belongs to $E$ also. 
This contradicts our assumption on $Y$.


\end{proof}

\begin{rema}\label{whynotgenerral}

The proof above doesn't work 
without countability restrictions since it does not seem possible to construct a similar inductive system. 
At each step of 
our construction we only have to construct a 
certain $Y_i$ from $Y_{i-1}$ via considering a certain cone.
However,  constructing the inductive system in 
the uncountable case would certainly require "completing" 
such a sequence of $Y_i$ 
by a certain "transfinite cone". The authors do not know how to achieve this unless $ \underline{C}$ possesses some sort of "enhancement" (for example, a differential graded one is certainly sufficient for these purposes since it can be used to construct certain "canonical cones" of morphisms that allow passing to "transfinite limits").
Possibly, this difficulty is related to the reason that persuaded the authors of \cite{postov} to consider derivators (in their Theorem 3.7).
\end{rema}

\section{Proof of the theorem in the general case}\label{prgen}

In this section we prove  Theorem \ref{extsunc} in the general case (though we have to consider the intermediate case of an at most countable $R$ first).

\subsection{"Approximating" categories by countable subcategories}\label{scountr}

In this section we introduce several constructions and techniques used to prove the main theorem.

Till the end of 
\S\ref{scountr} $R$ will be an at most countable ring.

We start with the following very easy "additive" lemma; next we apply it in the proof of a certain closely related "triangulated" statement (that is just a little more complicated).

\begin{lem}\label{preappr}
For any at most countable set of objects $O$ and at most countable set of morphisms $M$ between elements of $O$ in an (additive) $R$-linear category $\underline{B}$ there exists a (non-full!) $R$-linear subcategory $\underline{B}(M;O)$ of $\underline{B}$ 
whose set of objects and set of morphisms are both countable, such that $\operatorname{Obj} \underline{B}(M;O)$ contains $O$ and $\operatorname{Mor} \underline{B}(M;O)$ contains $M$.
\end{lem}
\begin{proof}
For any finite set $o_1,\cdots, o_n$ of elements of $O$ ($n\ge 0$) choose a representative in the isomorphism class of $\bigoplus\limits_{i=1}^n o_i$ in $\operatorname{Obj} \underline{B}$. Denote it by $S(o_1,\cdots,o_n)$;
  denote the countable set $\{S(o_1,\cdots, o_n) | n\ge 0, 
o_i \in O\}$ by $O'$. 

Now take  $\underline{B}(M;O)$ being the subcategory of $\underline{B}$ whose set of objects is $O'$ and the sets of morphisms 
 are as follows:
$\underline{B}(M;O)(o_1,o_2) = \{\sum\limits_{i=1}^n r_if^{n_i}_i\circ \cdots\circ f^{1}_i \} $; here we consider all composable chains (of lengths $n_i\ge 0$) of elements of $M$ such that the domain of $f^1_i$ is $o_1$, the codomain of $f^{n_i}_i$ is $o_2$, and $ r_i\in R $. By definition, the latter set is an $R$-submodule of $\underline{B}(o_1,o_2)$.
For any $o_1,\cdots,o_m, o'_1, \cdots, o'_n\in O$ (where $m,n\ge 0$) we set $\underline{B}(M;O)(\bigoplus\limits_{i=1}^m o_i, \bigoplus\limits_{i=1}^n o'_i)$ as the set of matrices whose $(i,j)$-entry belongs to $\underline{B}(M;O)(o_i,o'_j)$. 
The composition of morphisms in $\underline{B}$ certainly restricts to these sets; so 
we obtain an $R$-linear (additive) subcategory of $\underline{B}$.
\end{proof}

\begin{pr}\label{appr}
Let 
$ \underline{C}$ be a small $R$-linear triangulated category.
Then  for any  countable 
 $O\subset \operatorname{Obj}  \underline{C}$ and 
 any countable set of morphisms $M$ between elements of $O$ 
there exists an 
 $R$-linear triangulated subcategory $ \underline{C}_{tr}(M;O)$ of $ \underline{C}$  
whose set of objects and set of morphisms are both countable, such that $\operatorname{Obj}  \underline{C}_{tr}(M;O)$ contains $O$ and $\operatorname{Mor}  \underline{C}_{tr}(M;O)$ contains $M$.
\end{pr}
\begin{proof}
First we construct some countable sets $O^{(n)}, M^{(n)}$ of $ \underline{C}$-objects and morphisms for $n\ge 0$. 

The construction is inductive.
We start from  $O^{(0)} = O$ and $M^{(0)} = M$. 

To make the inductive step, for $n\ge 1$ we consider the corresponding  $(O^{(n-1)}, M^{(n-1)})$ and choose a countable $R$-linear subcategory \linebreak $ \underline{C}(O^{(n-1)}; M^{(n-1)})$ in $ \underline{C}$ (see Lemma \ref{preappr}). We  denote the  set of objects  and morphisms in  $ \underline{C}(O^{(n-1)}, M^{(n-1)})$   by $O'$ and $M'$, respectively. 
Next, for every 
$f \in M'$ choose a distinguished triangle containing it  and add  the corresponding (countable and shift-stable) data to $(O',M')$, 
obtaining certain sets  $O''$ and $M''$, 
respectively. 
 Finally, for every lower cap of a $ \underline{C}$-octahedral diagram whose 
morphisms belong $M''$  one chooses  an upper cap for it in $ \underline{C}$  and adds all the "new" objects and 
morphisms  
to $O''$ and $M''$, 
 respectively. 
We denote the sets obtained 
as a result  of  the latter series of operations by $O^{(n)}$ and $M^{(n)},$ respectively. 

The sets $\bigcup_{n\in \mathbb{N}} O^{(n)}$, $\bigcup_{n\in \mathbb{N}} M^{(n)}$ obviously give certain 
$R$-linear subcategory $ \underline{C}_{tr}(M;O)$ of $ \underline{C}$; the shift functors for $ \underline{C}$ (i.e.,  the functors $[n]$ for $n\in \z$) can certainly be restricted to it.
We define  the distinguished triangles in $ \underline{C}_{tr}(M;O)$ as
 those triangles of $ \underline{C}_{tr}(M;O)$-morphisms that are distinguished in $ \underline{C}$. By construction, all the axioms of triangulated categories are fulfilled in $ \underline{C}_{tr}(M;O)$ (we do not have to ensure the validity of the axiom TR3 in  $ \underline{C}_{tr}(M;O)$ by Lemma 2.2 of \cite{mayaddtr}). 
\end{proof}

\begin{rema}\label{mtheor}
1. Certainly, the proofs of Lemma \ref{preappr} and Proposition \ref{appr} are 
 closely related to the proof of the seminal L\"{o}wenheim-Skolem theorem. 

2. Certainly, we could have ensured the validity of the axiom TR3 in  $ \underline{C}_{tr}(M;O)$ "directly" (i.e., by adding certain morphisms to $M^{(n)}$ at each step). 
\end{rema}


Now we consider certain versions of "filtered products" of objects, categories, and functors.
Note first that any (small) product of triangulated categories has the natural structure of a triangulated category (resp. any small product of abelian categories is abelian). So we can introduce the following definition.

\begin{defi}\label{limcat}
Let $(I, \succeq  )$ be a directed set; let $\underline{A}_i,\ i\in I$, be triangulated (resp. abelian) categories. 

We define the reduced product  $\prod\limits_{(\to,i\in I)} \underline{A}_i$ as the Verdier localization (resp. the Serre localization) of the product  
$\prod\limits_{i\in I} \underline{A}_i$ by the full triangulated subcategory (resp. the full Serre subcategory) whose objects are 
the families   $(a_i)$, $a_i\in \operatorname{Obj} \underline{A}_i$, 
such that $a_n=0$ 
 for sufficiently large $n$. 

Certainly, if  all $\underline{A}_i$ are $A$-linear (for some commutative unital  ring $A$) then their reduced product is $A$-linear also.

In case  no confusion can arise we will just write $\prod\limits_{\to} \underline{A}_i$.
\end{defi}

\begin{rema}\label{limrem}
1. This reduced product can be characterized as the direct categorical limit $\ilim\limits_{n\in I} \prod\limits_{i\in I_n} \underline{A}_i$ (where $I_n$ is defined in  \S\ref{snotata}); here 
for $n' \succeq  n$ the corresponding functor  $\prod\limits_{i \in I_n} \underline{A}_i \to  \prod\limits_{i\in I_{n'}} \underline{A}_i$ is the natural projection.

2. A family  $(f_i) \in \prod\underline{A}_i(X,Y)$ yields a zero morphism in $\prod\limits_{\to} \underline{A}_i$ if and only if 
$f_n=0$ for sufficiently large $n$ (see \S\ref{snotata}).

In the case of triangulated $\underline{A}_i$ 
we have the following characterization of distinguished triangles in $\prod\limits_{\to} \underline{A}_i$: 
a family of morphisms $(X_n) \stackrel{f_n}\to (Y_n) \stackrel{g_n}\to (Z_n)\to (X_n[1])$ yields a distinguished triangle  if and only if 
$X_n \to Y_n \to Z_n \to X_n[1]$ are distinguished triangles for sufficiently large $n$.
Similarly, in the case  where $\underline{A}_i$ are abelian, 
a family of morphisms $(X_n) \stackrel{f_n}\to (Y_n) \stackrel{g_n}\to (Z_n)$ 
yields an exact sequence in $\prod\limits_{\to} \underline{A}_i$ if and only if 
$g_n\circ f_n=0$  for $n$ large enough and the corresponding morphisms $\operatorname{Im} f_n \to \operatorname{Ker} g_n$ are isomorphisms for $n$ large enough. The latter is equivalent to  $X_n \to Y_n \to Z_n$ being exact sequences for $n$ large enough.

3. Let $A$ be a (unital, commutative) ring and let $F_i: \underline{C}_i^{op} \to A-\operatorname{mod}$  be a family of $A$-linear cohomological functors (for some $A$-linear triangulated $ \underline{C}_i$, $i \in I$). 
Denote by $G$ the ($A$-linear) functor from the "reduced power"\ $\prod\limits_{(\to,i\in I)} A-\operatorname{mod}$  into $A-\operatorname{mod}$ that maps a family 
$(M_i\in \operatorname{Obj} A-\operatorname{mod})$ into  $\ilim\limits_{n\in I} \prod\limits_{i\in I_n} M_i$. 
Certainly, $G((M_i))=\{0\}$ 
if and only if $M_i = \{0\}$ for $i$ large enough.   
Denote by $F'$ the composition $\prod\limits_{\to}  \underline{C}_i^{op} \stackrel{\prod\limits_{\to} F_i}\to \prod\limits_{\to} A-\operatorname{mod} \stackrel{G}\to A-\operatorname{mod}$. 
Then we define the reduced product of 
 $F_i$ as the functor $\prod\limits_{\to}F_i :\prod\limits_{\to}  \underline{C}_i^{op} \to A-\operatorname{mod}$ such that 
the composition $\prod  \underline{C}_i^{op} \to \prod\limits_{\to} \underline{C}_i^{op} \to A-\operatorname{mod}$ is equal to $F'$ (which exists and is unique by the universal property of localization).

4. If $I$ has the greatest element $i^{max}$, then certainly  $ \prod\limits_{\to} \underline{A}_i\cong \underline{A}_{i^{max}}$. 
Thus the category $\mathfrak{C}$ defined in Lemma \ref{appremb} is isomorphic to  $ \underline{C}$ if $ \underline{C}$ is countable; if $R$ is countable, then the functor 
$F^Y$  that we construct in the proof of Theorem  \ref{extsunc} (in \S\ref{rcas} below) is isomorphic to the corresponding  $\hat{F}^Y_R$. 
\end{rema}

The following lemma is crucial for the proof of Theorem \ref{extsunc}.

Denote by $I$ be the set of countable $R$-linear triangulated subcategories (i.e., of  those subcategories whose set of objects  and set of morphisms are both countable) of $ \underline{C}$ ordered by (non-full) exact inclusions (it is a directed set thanks to Proposition \ref{appr}; recall that we assume $ \underline{C}$ to be a small category).

\begin{lem}\label{appremb}
Denote by $\mathfrak{C}$ the reduced product $\prod\limits_{(\to,C\in I)} C$. Then there exists a faithful exact $R$-linear functor $F:  \underline{C} \to \mathfrak{C}$  sending $X \in \operatorname{Obj}  \underline{C}$ into the family $(X_C)_{C\in I}$, where we set $X_C = X$ if $X \in \operatorname{Obj} C$ and $X_C = 0$ otherwise.
\end{lem}
\begin{proof}
For any morphism $f \in  \underline{C}(X, Y)$ we define $F(f)$ as the class of the family $(F_C(f))$, where we set  $F_C(f)=f$ if $f \in \mo C$ (for $C\in I$)  and $F_C(f)=0$ 
otherwise.
Now we check that this correspondence yields a functor. Take some composable morphisms $f$ and $g$ in $ \underline{C}$. 
By Remark \ref{limrem}(2), it suffices to show that $F_C(f) \circ F_C(g) = F_C(f\circ g)$ for $C$ large enough. 
By Proposition \ref{appr}, there exists a category $C'\in I$ containing $f$ and $g$ (and so, also $f\circ g$). Obviously, $F_C(f) \circ F_C(g) = F_C(f\circ g)$ for any $C \in I$ such that $C'$ is a (triangulated) subcategory of $C$.

Next, $F$ is certainly faithful,  $R$-linear, and respects shifts.

It remains to prove that $F$ sends distinguished triangles into distinguished triangles. 
Let $X \to Y \to Z \to X[1]$ be a distinguished triangle in $ \underline{C}$.
By Remark \ref{limrem}(2), it suffices to show that the triangles $F_C(X) \to F_C(Y) \to F_C(Z) \to F_C(X)[1]$ are distinguished  for $C$ large enough. 
By Proposition \ref{appr}, there exists a category $C'\in I$ containing the triangle $X\to Y\to Z \to X[1]$. Obviously,  the triangle $F_C(X) \to F_C(Y) \to F_C(Z) \to F_C(X)[1]$ is distinguished for any $C \in I$ such that $C'$ is a (triangulated) subcategory of $C$.
\end{proof}

\begin{coro}\label{funct}
Assume that for any $C\in I$ there exists an $R$-linear cohomological functor 
 $\mathcal{F}_C:C^{op}\to R-\operatorname{mod}$.
Then there exists an $R$-linear cohomological functor $\mathcal{F}: \underline{C}^{op} \to 
 R-\operatorname{mod}$ such that $\mathcal{F}(X)$ is zero if and only if $\mathcal{F}_C(X)=\{0\}$  for large enough  $C$  satisfying $X \in \operatorname{Obj} C$.
\end{coro}
\begin{proof}
Denote by $Q$ the faithful embedding of $ \underline{C}$ into the category $\mathfrak{C} = \prod\limits_{(\to,C\in I)} C$ (which exists by Lemma \ref{appremb}). 
Denote by $\mathcal{F}':\mathfrak{C}^{op} \to R-\operatorname{mod}$ the reduced product of cohomological functors $\mathcal{F}_C$ (see Remark \ref{limrem}(3)). 
Now we take $\mathcal{F}=\mathcal{F}'\circ Q$. 
By the definition of $\mathcal{F}'$, if $\mathcal{F}_C(X)=\{0\}$ 
for $C$ large enough then $\mathcal{F}'(Q(X))=\{0\}$. 
Conversely, assume $\mathcal{F}'(Q(X))=\{0\}$. Then $\mathcal{F}_C(Q(X))=\{0\}$  for $C$ large enough. By the definition of $Q$, we have $\mathcal{F}_C(Q(X)) = \mathcal{F}_C(X)=\{0\}$  for $C$ large enough (that satisfy $X \in \operatorname{Obj} C$).
\end{proof}


Now we can prove 
Theorem \ref{extsunc} under less restrictive countability assumptions.
Already the partial case $R = \mathbb{Z}$ of this statement is quite interesting and non-trivial.

\begin{pr}\label{rcountpr}
Theorem \ref{extsunc} is valid in the case where $R$ is at most countable.
\end{pr}

\begin{proof}
We define certain $R$-linear cohomological functors $F^Y_C$ from $C$ to $R-\operatorname{mod}$ for $C \in I$.
In the case $Y$ is not an object of $C$ we set $F^Y_C = 0$.
If $Y \in \operatorname{Obj} C$ we take $F^Y_C$ being a cohomological functor such that $F^Y_C|_{E\cap \operatorname{Obj} C} =0$ and $F^Y_C(Y)\neq \{0\}$ 
(as constructed in Proposition \ref{exts}).

Applying Corollary \ref{funct}, we obtain an 
$R$-linear cohomological functor $F^Y: \underline{C}^{op} \to R-\operatorname{mod}$ such that $F^Y(X)$ is zero if and only if $F^Y_C(X)=\{0\}$  for $C$ large enough  (and $X \in \operatorname{Obj} C$). 
Certainly, it follows that $F^Y|_E = 0$. 
Moreover, if  $F^Y(Y)=\{0\}$ then $F^Y_C(Y) = \{0\}$ for some $C$ containing $Y$. Thus $Y$ belongs to $E\cap \operatorname{Obj} C \subset E$, which contradicts our assumption on $Y$.

\end{proof}

\subsection{The  case of an uncountable $R$}\label{rcas}

\begin{defi}\label{redprod}
For a directed set $I$ and a family $(M_i:\ i\in I)$
of objects of some (closed with respect to arbitrary 
filtered limits and colimits) abelian category 
  we define the reduced product $\prod\limits_{(\to,i\in I)} M_i$ 
	 as $\ilim\limits_{k\in I} \prod\limits_{i\in I_k} M_i$ 

\end{defi}

Certainly, the reduced product can also be characterized as the cokernel of the natural 
	 morphism $\coprod\limits_{k\in I}(\prod\limits_{k\not\succeq  n} M_n) \to \prod\limits_{n\in I} M_n$ since  both of these have the same universal property. 

\begin{pr}\label{propert}
Let $M_i\in \operatorname{Obj} \underline{A}$ for ${i\in I}$, 
where $I$ is a directed set and $\underline{A}$ is an abelian category.
Assume for any $i\in I$ the object $M_i$ is equipped with  an action of 
 a commutative unital ring 
$R_i$, where $(R_i, \phi_{j,i})$ is an inductive system of (
unital) rings (with connecting homomorphisms $ \phi_{j,i}:R_j\to R_i$ corresponding to $\preceq$). 

1. Then there is a natural action of $\ilim\limits_{i \in I} R_i$ on $\prod\limits_{(\to,i\in I)} M_i$. 

2. Let $(L_i)_{i\in I}$ be objects of $\underline{A}$ that are 
equipped with   $R_i$-actions  also (for all ${i\in I}$), 
and let $f_i:M_i \to L_i$ be  a family of morphisms that respect the corresponding actions. 
Then 
 the corresponding morphism $f:\prod\limits_{\to} M_i 
\to \prod\limits_{\to} L_i$ respects the  $\ilim\limits_{i \in I} R_i$ -action.
\end{pr}
\begin{proof}
1.
It suffices to construct ring homomorphisms $R_i \stackrel{g_i}\to \operatorname{End}_{\underline{A}}(\prod\limits_{(\to,n\in I)} M_n)$ such that the composition
 $R_j \stackrel{\phi_{j,i}}\to R_i \stackrel{g_i}\to \operatorname{End}_{\underline{A}}(\prod\limits_{\to} M_n)$ equals  $g_j$ for any $i\succeq  j \in I$ (i.e., to verify that these actions are compatible).

For any $i \succeq  j \in I$ we have an action of $R_j$ on $M_i$ via the composition $R_j \to R_i \to \operatorname{End}_{\underline{A}}(M_i)$. 
This yields a diagonal action of $R_j$ on $\prod\limits_{k\in I_n} M_k$ for any $n\succeq  j$. 
Moreover, if $n\succeq  i$, this action is compatible with the corresponding action of $R_i$.
So we obtain an action $g_j$ of $R_j$ on $\ilim\limits_{n\in I_j} \prod\limits_{I_n} M_k \cong \prod\limits_{\to} M_n$. 
By construction, the action of $R_j$ on  $\ilim\limits_{n\in I_i}  \prod\limits_{I_n} M_k \cong \ilim\limits_{n\in I_j} \prod\limits_{I_n} M_k \cong \prod\limits_{\to} M_n$ is compatible with the corresponding action of $R_i$.

2. To prove the 
assertion 
 it suffices to verify that $f$ respects the action of $R_i$ for any $i\in I$. 
Certainly, the morphism $\ilim\limits_{n\in I_i} \prod\limits_{I_n} M_k \to \ilim\limits_{n\in I_i} \prod\limits_{I_n} N_k$ respects the action of $R_i$ since it is induced by the $R_i$-linear morphisms $\prod\limits_{I_n} M_k \to \prod\limits_{I_n} N_k$ (for $n \in I_i$).
\end{proof}

Now we are able to conclude the proof of Theorem \ref{extsunc}.

\begin{proof}[Proof of Theorem \ref{extsunc} for an arbitrary $R$]

Note that $R$ is equal to  the inductive limit of its (at most) countable unital subrings; thus $R = \ilim R_i$ for some directed set of at most countable (commutative unital) rings $R_i$. Next, $\underline{C}$ is $R_i$-linear for each of these $R_i$.
Hence (by Proposition \ref{rcountpr}) there exist $R_i$-linear cohomological functors $\hat{F}^Y_{R_i}: \underline{C}^{op} \to R_i-{\operatorname{mod}}$ such that $\hat{F}^Y_{R_i}|_E = 0$ and $\hat{F}^Y_{R_i}(Y)\neq \{0\}$. 
We take $F^Y_{R_i} = G_{R_i}\circ \hat{F}^Y_{R_i}: \underline{C}^{op} \to \ab$, where $G_{R_i}$ is the corresponding forgetful functor.

Now denote by $\hat{F}^Y$  the reduced product of cohomological functors $\prod\limits_{\to} F^Y_{R_i}$ (see Remark \ref{limrem}(3)). By Proposition \ref{propert}, this functor factors through 
$ R-\operatorname{mod}$. Denote the corresponding functor $ \underline{C}^{op} \to R-\operatorname{mod}$ by $F^Y$. 
Since $F^Y$ is $R_i$-linear for any $R_i$, $F^Y$ is also $R$-linear. 
By Remark \ref{limrem}(3), $F^Y(X)$ is zero if and only if $\hat{F}^Y_{R_i}(X)=\{0\}$  for $R_i$ large enough. 
Hence $F^Y|_E=0$ and $F^Y(Y)\neq\{0\}$. 
\end{proof}


As a corollary we present a list of 
criteria for $E$ to be Karoubi-closed  and extension-closed in $\cu$.

\begin{coro}\label{mcoro}
Let $E$ be a subset of $\operatorname{Obj}  \underline{C}$.
Then the following conditions are equivalent.

1. There exists a set $\{F_i\},\ i\in I$, of cohomological functors from $ \underline{C}$ with values in some abelian categories $\underline{A}_i$ such that the set $\{c\in \operatorname{Obj}  \underline{C}: F_i(c)
=0\ \forall i\in I\}$ equals  $E$.

2. $E$ is an extension-closed Karoubi-closed subset of $\operatorname{Obj}  \underline{C}$.

3. There exists a set $\{F_i\}$ of $R$-linear cohomological functors on $ \underline{C}$ with values in $R-\operatorname{mod}$ such that the set $\{c\in \operatorname{Obj}  \underline{C}: F_i(c)= \{0\}\}$ equals  $E$.

4. There exists a single $R$-linear cohomological functor $F: \underline{C}^{op}\to  R-\operatorname{mod}$ such that the set $\{c\in \operatorname{Obj}  \underline{C}: F(c)= \{0\}\}$ equals  $E$.

5. Denote by $H: \underline{C}\to \underline{A}( \underline{C})$ the universal homological functor for $ \underline{C}$ ($\underline{A}( \underline{C})$ is the  abelianization of $ \underline{C}$ and we call $H$ a homological functor since it is covariant; see  Appendix A in \cite{krause}).
Denote by $\underline{A}(E)$ the Serre subcategory of $\underline{A}( \underline{C})$ generated by 
$H(E)$. Then $H(\operatorname{Obj}  \underline{C})\cap \operatorname{Obj} \underline{A} (E) = H(E)$.

\end{coro}
\begin{proof}
Obviously, (1) $\Rightarrow$ (2), (3) $\Rightarrow$ (1), and (4) $\Rightarrow$ (3). 
The implication (2) $\Rightarrow$ (3) follows from Theorem \ref{extsunc}. 

The composition functor $H_E: \underline{C} \to \underline{A}( \underline{C})\to \underline{A}( \underline{C})/\underline{A}(E)$ (the latter category is the corresponding Serre localization) yields a homological functor on $ \underline{C}$ whose set of zeros is $H(\operatorname{Obj}  \underline{C})\cap \operatorname{Obj}\underline{A}(E)$. Consider the opposite to this functor (i.e., the corresponding $H_E^{op}: \underline{C}^{op}\to (\underline{A}( \underline{C})/\underline{A}(E))^{op}$; this is a cohomological functor with values in $ (\underline{A}( \underline{C})/\underline{A}(E))^{op}$), we obtain (5) $\Rightarrow$ (1). Conversely, for any set of cohomological functors $\{F_i\}$ with values in $\underline{A}_i$ such that the set $\{c\in \operatorname{Obj}  \underline{C}: F_i(c)
=0\ \forall i\in I\}$ equals  $E$ there exist exact functors $G_i:(\underline{A}( \underline{C})/\underline{A}(E))^{op} \to \underline{A}_i$ such that $G_i\circ H_E^{op} = F_i$ (by Lemma A.2 of \cite{krause}).  Hence the set  $\{c\in \operatorname{Obj}  \underline{C}: H_E(c)=0\} = H(\operatorname{Obj}  \underline{C}^{op})\cap \operatorname{Obj}\underline{A}(E)$ is a subset of $\{c\in \operatorname{Obj}  \underline{C}: F_i(c)=0\ \forall i\in I\} = E$. Thus $H(\operatorname{Obj}  \underline{C}^{op})\cap \operatorname{Obj}\underline{A}(E)$ equals to $H(E)$, and we obtain condition (5).

So it remains  to prove (3) $\Rightarrow$ (4). 
Assume (3) to be fulfilled. Certainly, the product functor $F(-) = \prod\limits_{Y \in (\operatorname{Obj}  \underline{C}) \setminus E} F^Y(-)$ is $R$-linear and cohomological. Moreover, $F(X) =\{0\}$ if and only if $F^Y(X)  =\{0\}$ for every $Y \in (\operatorname{Obj}  \underline{C}) \setminus E$. Thus $F(X) =\{0\}$ if and only if $X \in E$.
\end{proof}

\section{An application: computing envelopes via "localizing  the coefficients"}\label{applc}

For any prime ideal $p \triangleleft R$ denote the localization of $R$ at $p$ by $R_p$.
Consider the  category $ \underline{C}_p$ whose objects are those of  $ \underline{C}$ and whose morphisms are obtained by tensoring the $ \underline{C}$-ones by $R_p$ (over $R$).
Denote by  $L_p$ the obvious functor $ \underline{C}\to  \underline{C}_p$. We recall some well-known properties of this construction (which were studied in detail in Appendix A.2 of \cite{kellyth} in the case $R=\z$).

\begin{pr}\label{trtens}
$ \underline{C}_p$ is a triangulated category, whereas $L_p$ is isomorphic to the Verdier quotient  functor for the localization of $ \underline{C}$ by 
 its triangulated subcategory  generated by  $\{\operatorname{Cone}(X \stackrel{s\operatorname{id}_X}\to X) | s\in R\setminus p, X\in \operatorname{Obj}  \underline{C}\}$.
\end{pr}
\begin{proof}
See Proposition B.1.5 of \cite{cdet}.

\end{proof}

Now we 
are able to prove  Corollary \ref{clocoeff}.

Certainly, the "only if" part of the statement is obvious. 

Now let us prove the converse implication. 
Denote by $E_p$ the envelope of $L_p(D)$ in $ \underline{C}_p$ (and recall that $E$ denotes the $ \underline{C}$-envelope of $D$).
Let us assume that $L_p(Y)$ belongs to $E_p$ for every maximal ideal $p$ of $R$.

Assume first that $ \underline{C}$ is small.
By 
Corollary \ref{mcoro},
there exists an $R$-linear  cohomological functor $F: \underline{C}^{op}\to  R-\operatorname{mod}$ 
 such that $F(X) = \{0\}$ if and only if $X \in E$. 
Certainly, for any $p$ (that is a maximal ideal of $R$) the correspondence $F_p:X\mapsto F(X)\otimes_R R_p$ yields a cohomological functor on $ \underline{C}_p$. 
Since $F_p|_{L_p(D)} = 0$, we also have $F_p|_{E_p} = 0$.
Since for any 
maximal ideal $p$ of $R$ the object $L_p(Y)$ belongs to $E_p$, we have $F_p(Y) =F(Y)\otimes_R R_p= \{0\}$ (for any 
$p$). 
 Thus $F(Y) = \{0\}$; hence  $Y$ belongs to $E$.

It remains to  deduce the general case of the corollary from the "small" one. 
Certainly, for $p$ running through maximal ideals of $R$ there exist finite sets  $D_p\subset D$ 
such that $L_p(Y)$ belongs to the $
 \underline{C}_p$-envelope of $L_p(D_p)$. Next, there exists a small full triangulated  $R$-linear subcategory $ \underline{C}'$ of $ \underline{C}$ whose set
of objects contains $(\cup_p D_p)\cup \{Y\}$. Denote $E\cap \operatorname{Obj}  \underline{C}'$ by $E'$; for any maximal $p\trianglelefteq R$ denote 
the corresponding localization of coefficients at $p$  functor $\underline{C}'\to \underline{C}'_p$ by $L'_p$.
Since $\underline{C}'_p$ is a full subcategory of $  \underline{C}_p$ for any $p$, $L'_p(Y)$ belongs to
the $\underline{C}'_p$-envelope of $L'_p(E)$ (note that the latter set contains the $\underline{C}'_p$-envelope of $L'_p(D_p)$).
 Since our corollary is valid for  the triple $(\underline{C}',E', Y)$, we obtain that $Y$ belongs to $E'\subset E$.

\begin{rema}\label{runion}
  In the "motivic" applications the authors have in mind the usage of Gabber's resolution of singularities results (see the Introduction) does not "naturally" yield a single $D$ for all $p$. Instead, for any maximal ideal $p$ of $R$ we can find certain $D_p\subset \operatorname{Obj}  \underline{C}$ such that $L_p(Y)$ belongs to the $
 \underline{C}_p$ envelope of $L_p(D_p)$. Yet this certainly implies (somewhat similarly to the reasoning above) that $Y$ belongs to the $ \underline{C}$-envelope of $ \cup_{p}D_p$.

\end{rema}

\end{document}